\documentclass[11pt]{amsart}
\usepackage{amssymb}
\usepackage{graphicx}
\usepackage{graphics}
 \usepackage[mathscr]{euscript} 
\usepackage{amscd}
\usepackage{enumerate}
\usepackage[all]{xy}
\usepackage{mathtools} 
\usepackage{color}
\usepackage[table]{xcolor}
\usepackage{xspace} 
\usepackage[margin=3cm, marginparwidth=2.5cm]{geometry}
\usepackage{hyperref}
\usepackage[normalem]{ulem}

\newtheorem{thm}{Theorem}[section]
\newtheorem{theorem}[thm]{Theorem}

\newtheorem{corollary}[thm]{Corollary}

\newtheorem{lemma}[thm]{Lemma}

\newtheorem{proposition}[thm]{Proposition}

\newtheorem{question}[thm]{Question}
\theoremstyle{definition}

\newtheorem{definition}[thm]{Definition}

\theoremstyle{remark}

\newtheorem{remark}[thm]{Remark}

\newtheorem{example}[thm]{Example}

\newcommand{\mc}{\mathcal}

\renewcommand{\theta}{\vartheta}

\newcommand{\set}[1]{\left\{#1\right\}}

\newcommand{\scal}[2]{{\langle #1, #2 \rangle}}

\newcommand{\de}{\partial}
\newcommand{\del}{\partial}

\DeclareMathOperator{\im}{im}
\DeclareMathOperator{\Int}{Int}

\newcommand{\Z}{\mathbb{Z}}

\newcommand{\F}{\mathbb{F}}

\newcommand{\nbd}{\mathcal{N}} 
\newcommand{\A}{\mathbb{A}}
\newcommand{\CP}{\mathbb{CP}^2}
\newcommand{\bCP}{\overline{\mathbb{CP}^2}}

\DeclareMathOperator{\HF}{HF}

\newcommand{\HFi}{\HF^\infty}
\newcommand{\HFp}{\HF^+}
\newcommand{\HFm}{\HF^-}

\newcommand{\HFred}{\HF_{\mathrm{red}}}

\DeclareMathOperator{\CFK}{CFK}
\newcommand{\CFKi}{\CFK^\infty}

\newcommand{\red}{\operatorname{red}}

\newcommand{\w}{\mathbf{w}}
\newcommand{\x}{\mathbf{x}}

\newcommand{\z}{\mathbf{z}}

\newcommand{\gr}{\mathbf{gr}}

\DeclareMathOperator{\Spin}{Spin}

\newcommand{\spinc}{spin$^c$\xspace}
\newcommand{\Spinc}{\Spin^c}

\newcommand{\s}{\mathfrak{s}}

\renewcommand{\t}{\mathfrak{t}}

\newcommand{\Arf}{\operatorname{Arf}}
\newcommand{\Slice}{\mathscr{S}}
\newcommand{\Knots}{\mathscr{K}}
\newcommand{\Xo}{X^{\circ}}

\newcommand{\st}{\operatorname{st}}
\newcommand{\BF}{\mathit{BF}}

\newcommand{\mK} {\mkern3mu \overline{\mkern-3mu K \mkern-1mu} \mkern1mu}  
\newcommand{\mX} {\mkern3mu \overline{\mkern-3mu X \mkern-1mu} \mkern1mu}

\newcommand{\KT}{\mathrm{K3}}
\newcommand{\Wh}{\mathrm{Wh}}
\newcommand{\Tors}{\mathrm{Tors}}
\DeclareMathOperator{\rk}{rank}
\DeclareMathOperator{\coker}{coker}
\newcommand{\Tower}{\mathcal{T}}
\newcommand{\Se}{S^1_{!}}
\newcommand{\ks}{\operatorname{ks}}

\newcommand{\KDV}{K_{\operatorname{DV}}}
\DeclareMathOperator{\dsh}{D}
\newcommand\intB{\smash{\mathring{B}^4}}
\newcommand\slbar{\overline{\mathit{sl}}}

\begin{document}

\title[Relative genus bounds in indefinite four-manifolds]{Relative genus bounds in indefinite four-manifolds}

\author{Ciprian Manolescu}%
\address {Department of Mathematics, Stanford University, 450 Jane Stanford Way\\ 
Stanford, CA 94305, USA}
\email {\href{mailto:cm5@stanford.edu}{cm5@stanford.edu}}
\thanks{CM was supported by NSF grant DMS-2003488 and a Simons Investigator Award.}

\author{Marco Marengon}%
\address{HUN-REN Alfr{\'e}d R{\'e}nyi Institute of Mathematics, R{\'e}altanoda utca 13-15\\ 1053 Budapest, Hungary}%
\email{\href{mailto:marengon@renyi.hu}{marengon@renyi.hu}}%
\thanks{MM was supported by NSF FRG grant DMS-1563615 and the Max Planck Institute for Mathematics.}

\author{Lisa Piccirillo}%
\address{Department of Mathematics, University of Texas at Austin, 2515 Speedway\\ 
Austin, TX 78712, USA}%
\email{\href{mailto:lisa.piccirillo@austin.utexas.edu}{lisa.piccirillo@austin.utexas.edu}}
\thanks{LP was supported by NSF postdoctoral fellowship DMS-1902735 and the Max Planck Institute for Mathematics} 


\begin{abstract}
Given a closed four-manifold $X$ with an indefinite intersection form, we consider smoothly embedded surfaces in $X \setminus \intB$, with boundary a knot $K \subset S^3$. We give several methods to bound the genus of such surfaces in a fixed homology class. Our tools include adjunction inequalities and the $10/8 + 4$ theorem. In particular, we present obstructions to a knot being H-slice (that is, bounding a null-homologous disk) in a four-manifold  and show
that the set of H-slice knots can detect exotic smooth structures on closed $4$-manifolds.
\end{abstract}
\maketitle

\section{Introduction}
A fundamental problem in four-dimensional topology is to find the minimal genus of embedded surfaces in a four-manifold, in a given homology class. For example, the Thom conjecture \cite{KM-thom} and the symplectic Thom conjecture \cite{os-thom} were problems of this type; their solutions rank among the major successes of gauge theory.
A relative version of the same problem concerns bounding the genus of properly embedded surfaces $\Sigma$ in a four-manifold $W$ with boundary, when $\del \Sigma$ is a given knot $K \subset \del W$ and the relative homology class of $\Sigma$ is fixed. 

We will focus on the case where the four-manifold has boundary $S^3$. We let $X$ be a closed, connected, oriented, smooth four-manifold, and consider properly embedded surfaces in 
$\Xo := X\setminus \intB,$
with boundary a classical knot $K \subset S^3$. One problem of interest is whether $K$ bounds a null-homologous disk in $\Xo$; if so, we say that $K$ is {\em H-slice in $X$}.

When $X=S^4$, the problem reduces to the well-known question of finding the four-ball genus of knots, and in particular of determining which knots are slice. More generally, when $X$ has definite intersection form, many of the gauge theoretic techniques for bounding the genus of embedded surfaces still apply; see \cite{os-tau, km, HR}. When $X =  \#^n \CP$ or $\#^n \bCP$, there are also bounds from Khovanov homology \cite{MMSW}.

Less is known about relative genus bounds in more complicated, indefinite four-manifolds, such as the K3 surface or complex surfaces of general type. Classical methods produce topological constraints (that apply equally well for surfaces embedded in a locally flat way in a topological four-manifold). We will review these in Section~\ref{sec:top}. They include constraints from the Arf invariant \cite{robertello, FK, KirbyTaylor}, from the Tristram-Levine signatures \cite{ConwayNagel}, and from a theorem of Rokhlin \cite{Rokhlin}.

The main purpose of this paper is to use gauge theory and Heegaard Floer homology to develop new techniques for bounding the genus of smoothly embedded surfaces with boundary, in indefinite four-manifolds.

Inside $4$-manifolds with non-trivial Seiberg-Witten (or Ozsv\'ath-Szab\'o) invariants, the genus of closed surfaces can be bounded using the adjunction inequalities from \cite{KM-thom}, \cite{MST}, \cite{os-thom}, \cite{os4}. This can be leveraged to bound the genus of surfaces $\Sigma \subset \Xo$ with boundary a knot $K$: by capping off $\Sigma$ with a smooth surface $F$ in some manifold $Z$ with $\del Z\supset S^3$ and $\del F =\mK$ (the mirror of $K$), we can apply the adjunction inequality in $\Xo \cup Z$ to the resulting closed surface $\Sigma \cup F$. 

The simplest way to do this is to take $Z=B^4$ and $F$ a 
surface in $B^4$ with boundary $\mK$. 
This gives an inequality of the following form, involving the slice genus $g_4(K)$ of the knot:
\begin{equation}
\label{eq:adjunction-g4}
\scal{c_1(\s)}{[\Sigma]} + [\Sigma]^2 \leq 2g(\Sigma)-2 + 2g_4(K).
\end{equation}
(See Theorem~\ref{thm:adjunction-g4} below.) Here and throughout the paper, $[\Sigma]^2$ is defined by identifying the two homology groups $H_2(\Xo, \del\Xo) \cong H_2(X)$; that is, $[\Sigma]^2 = [\Sigma \cup F]^2.$ 

A stronger bound can be obtained by taking $Z$ to be the cobordism associated to some surgery on $K$, and letting $F$ be the core of the $2$-handle. This improves the inequality \eqref{eq:adjunction-g4} to one involving the concordance invariant $\nu^+$ from knot Floer homology \cite{HW}. We obtain the following relative adjunction inequality (see Section \ref{sec:adj} for an introduction to $\Phi_{X, \s}$ and manifolds of Ozsv\'ath-Szab\'o simple type):
\begin{theorem}
\label{thm:adjunction-nu+}
Let $X$ be a closed 4-manifold, with $b_2^+(X)>1$. Let $\Sigma \subset \Xo$ be a smooth, properly embedded surface with $g(\Sigma) > 0$ and $\del \Sigma = K$, and let $\mK$ denote the mirror of $K$. Suppose that either  $[\Sigma]^2 \geq 2\nu^+(\mK)$ or $X$ is of Ozsv\'ath-Szab\'o simple type. Then, for every \spinc structure $\s \in \Spinc(X)$ for which the mixed invariant $\Phi_{X, \s}$ is non-zero, we have
\[
\scal{c_1(\s)}{[\Sigma]} + [\Sigma]^2 \leq 2g(\Sigma)-2 + 2\nu^+(\mK).
\]
\end{theorem}

\begin{remark}
A different relative adjunction inequality, in terms of the concordance invariant $\tau$, was proved by Ozsv\'ath and Szab\'o \cite[Theorem 1.1]{os-tau}. Theirs applies to surfaces in negative definite $4$-manifolds, whereas ours is for $4$-manifolds with $b_2^+(X) > 1$. 
\end{remark}


Theorem~\ref{thm:adjunction-nu+} gives non-trivial obstructions for surfaces $\Sigma$ of positive genus and with $[\Sigma]\neq 0$, but does not say anything about H-sliceness in $X$. Instead, we can get constraints on H-sliceness by filling $\Xo$ with suitable symplectic manifolds, and using adjunction inequalities based on the Bauer-Furuta invariants:

\begin{theorem}
\label{thm:symp}
Let $X$ and $X'$ be closed symplectic $4$-manifolds satisfying $b_2^+(X) \equiv b_2^+(X') \equiv 3 \!\! \pmod{4}$. Suppose that a knot $K \subset S^3$ is such that the mirror $\mK$ bounds a smooth, properly embedded disk $\Delta \subset \Xo$ with $[\Delta]^2 \geq 0$ and $[\Delta] \neq 0$. Then $K$ is not H-slice in $X'$.
\end{theorem}

From here we obtain the following application.

\begin{corollary}
\label{cor:exotica}
There exist smooth, homeomorphic four-manifolds $X$ and $X'$ and a knot $K \subset S^3$ that is H-slice in $X$ but not in $X'$. For example, one can take $$X=\#3\CP \# 20 \bCP, \ \ X' = K3 \# \bCP,$$ and $K$ to be the right-handed trefoil. 
\end{corollary}

This result sheds some light on the following well-known proposed strategy to disprove the smooth $4$-dimensional Poincar\'e conjecture: find a knot $K$ that is H-slice (or equivalently slice) in a homotopy $4$-sphere but not in $S^4$; see for example \cite{mnm}. Corollary~\ref{cor:exotica} gives the first example showing that indeed there are closed $4$-manifolds for which the set of H-slice knots can detect exotic smooth structures. (We note that the literature already contains examples of exotic $4$-manifold pairs, where the boundary $Y$ is \emph{not} $S^3$, such that some knot in $Y$ bounds a null-homologous smooth disk in one manifold and not in the other; this is the case, for instance, with Akbulut's corks \cite{AkbulutCork}).

We also observe that Corollary~\ref{cor:exotica} gives an example of a knot that is topologically but not smoothly H-slice in an indefinite $4$-manifold; under the homeomorphism $X \to X'$, the image of the smooth H-slice disk for the right hand trefoil in $X$ is a topological H-slice disk for the right-handed trefoil in $X'$. Since the right-handed trefoil does not bound a topological disk in $S^3\times [0,1]$, Corollary~\ref{cor:exotica} demonstrates that this disparity between smooth and topological sliceness is inherent to $X'$ rather than inherited from the well-known disparity between smooth and topological sliceness in $B^4$. 

In a different direction, Furuta's celebrated $10/8$-theorem \cite{Furuta} gives constraints on the intersection forms of smooth spin $4$-manifolds. Donald and Vafaee \cite{DonaldVafaee} used the $10/8$-theorem to derive a new sliceness obstruction (in the four-ball). Their result was strengthened by Truong in \cite{Truong}, by applying a refinement of Furuta's theorem (called the $10/8+4$ theorem) due to Hopkins-Lin-Shi-Xu \cite{HLSX}.

The same techniques can be used to obstruct H-sliceness in other $4$-manifolds. We obtain:

\begin{theorem}
\label{thm:DV}
Let $K \subset S^3$ be an $H$-slice knot in a closed spin 4-manifold $X$, and let $W$ be a spin 2-handlebody with $\de W = S^3_0(K)$. If $b_2(X) + b_2(W) \neq 1, 3, 23$, then
\[
b_2(X) + b_2(W) \geq \frac{10}{8} \cdot |\sigma(X) - \sigma(W)| + 5.
\]
\end{theorem}

In \cite[Section 3.2]{DonaldVafaee}, Donald and Vafaee applied their methods to show that a certain topologically slice knot, which we call $\KDV$ and reproduce in Figure \ref{fig:DV}, is not smoothly slice. Theorem~\ref{thm:DV} implies the following.

\begin{corollary}
\label{cor:DV}
The topologically slice knot $\KDV$ is not $H$-slice in the K3 surface.
\end{corollary}

Note that this example is qualitatively different from that of the trefoil in $X'=K3 \# \bCP$, because $\KDV$ bounds a locally flat disk in a neighborhood of the boundary, i.e., in $S^3 \times [0,1]$. In the terminology of \cite{KlugRuppik}, $\KDV$ is topologically shallow slice in the K3 surface, whereas the trefoil is topologically deep slice in $X'$.

\begin{remark}
In unpublished work, Anthony Conway and Oliver Singh used the same technique to investigate topological versus smooth H-sliceness in $\#^n (S^2\times S^2)$. 
\end{remark}

\subsection{Conventions} All manifolds will be assumed to be compact and oriented. We also assume the manifolds and the embeddings of surfaces to be smooth unless otherwise stated.

\subsection{Organization of the paper} In Section~\ref{sec:slice} we discuss the notions of slice and H-slice knots in four-manifolds, and give some examples. In Section~\ref{sec:top} we review several topological constraints on the existence of surfaces with boundary inside four-manifolds. In Section~\ref{sec:adj} we prove the relative adjunction inequality, Theorem~\ref{thm:adjunction-nu+}, along with Theorem~\ref{thm:symp} and Corollary~\ref{cor:exotica}. In Section~\ref{sec:DVT} we present the relative Donald-Vafaee obstruction, Theorem~\ref{thm:DV}, and prove Corollary~\ref{cor:DV}. Finally, in Section~\ref{sec:open} we list a few open problems.

\subsection{Acknowledgments} We are grateful to Matt Hedden, Michael Klug, Maggie Miller, Benjamin Ruppik, and Ian Zemke for helpful conversations. We also thank Gordana Mati{\'c}, Anubhav Mukherjee, Kouichi Yasui, and the referee for comments on a previous version of the paper.

\section{Slice and H-slice knots}
\label{sec:slice}
 Let $X$ be a closed, connected, smooth four-manifold, and consider properly embedded surfaces in 
$\Xo := X\setminus \intB.$

\begin{definition}
(a) We say that a knot $K$ in $S^3 \cong \del \Xo$ is {\em slice in $X$} if it bounds a smoothly, properly embedded disk $\Delta  \subset \Xo$. 

(b) If $K$ is slice in $X$ and the disk $\Delta$ can be taken so that $[\Delta]=0 \in H_2(\Xo, \del \Xo) \cong H_2(X),$ we say that $K$ is {\em $H$-slice in $X$.} 
\end{definition}
As a simple observation, we note that a knot $K$ is slice (resp.\ H-slice) in $X$ if and only if its mirror image $\mK$ is slice (resp.\ H-slice) in $\mX$.

We denote by $\Slice(X)$, resp.\ $\Slice_H(X)$, the set of knots that are slice, resp.\ H-slice, in $X$. In particular, we write $\Slice= \Slice(S^4)=\Slice_H(S^4)$ for the usual set of slice knots. We also write $\Knots$ for the set of all knots. Note that for every $X$ we have
\begin{equation}
\label{eq:slice}
\Slice \subseteq \Slice_H(X) \subseteq \Slice(X) \subseteq \Knots.
\end{equation}

In a topological $4$-manifold $X$, we also have related notions of {\em topologically slice}, and {\em topologically H-slice} knots in $X$, referring to the existence of disks that are embedded in a locally flat way. 

Let us mention a few results about sliceness and H-sliceness in some particular four-manifolds:
Norman \cite{Norman} and Suzuki \cite{Suzuki} proved that every knot is slice in  $\CP \# \bCP$ and $S^2 \times S^2$. Further, Cochran-Orr-Teichner [\cite{COT} Remark 1.3] showed that every knot of Arf invariant zero is H-slice in $\#^n(S^2 \times S^2)$ for some $n \geq 0$. At the opposite end, for manifolds such as $X= S^1 \times S^3$ or $T^4$, we have $\Slice(X) = \Slice_H(X)=\Slice$.  This can be seen by viewing $X$ as $B^4$ with some boundary identifications, then passing to some sufficiently large finite (self) cover $\tilde{X}$ (which we can still view as $B^4$ with some boundary conditions) in which the slice disk does not intersect the boundary of $B^4$. 

A more interesting example is $\CP$, for which all the inclusions in \eqref{eq:slice} are strict.
Indeed, there exist knots that are not slice in $\CP$, such as $T_{2,-15}$, cf.\ Yasuhara \cite{YasuharaExample, yasuhara}; there are also knots that are slice in $\CP$ but not H-slice in $\CP$, such as the left-handed trefoil; and knots that are H-slice in $\CP$ but not slice in $S^4$, such as the right-handed trefoil. (See Examples~\ref{ex:t1}--\ref{ex:t3} below.)  Sliceness and related notions in $\#^n \CP$ (or its reverse) were further investigated in \cite{CochranHarveyHorn}, \cite{Positive}, \cite{Pichelmeyer}. 

Invariants from Floer homology and Khovanov homology can be used to obstruct H-sliceness in definite four-manifolds \cite{os-tau, km, MMSW, HR}.

\begin{example}
\label{ex:t1}
The left hand trefoil (LHT) is not H-slice in $\CP$, for example by the adjunction inequality for $\tau$ or $s$~\cite{os4, MMSW}.
\end{example}

\begin{example}
\label{ex:RHT}
On the other hand, the right-handed trefoil (RHT) is H-slice in $\CP$. One way to see this is to consider the standard handle diagram for $\CP$. After we remove the 0-handle and the 4-handle, we get a cobordism $\CP\setminus(\intB \sqcup \intB)$ from $S^3$ to $S^3$. Observe that there is a null-homologous annulus in $\CP\setminus(\intB \sqcup \intB)$ from LHT in $\partial^-(\CP\setminus(\intB \sqcup \intB))$ to the curve $\gamma$ in $\partial^+(\CP\setminus (\intB \sqcup \intB))$ shown in the left frame of Figure \ref{fig:RHTinCP}. The annulus is null-homologous because $\gamma$ has vanishing linking number with the 2-handle. Now observe that when we identify $\partial^+(\CP\setminus(\intB \sqcup \intB))$ with the standard diagram of $S^3$, as in the right of Figure \ref{fig:RHTinCP}, we can identify $\gamma$ as the unknot. Since the unknot bounds a disk in the 4-handle, we have found a nullhomologous disk in $\CP\setminus \intB$ with boundary LHT in $\partial^-(\CP\setminus \intB)$. Adjusting for the standard outward-normal-first orientation on boundaries, the claim follows.
\begin{figure}
\begingroup%
  \makeatletter%
  \providecommand\color[2][]{%
    \errmessage{(Inkscape) Color is used for the text in Inkscape, but the package 'color.sty' is not loaded}%
    \renewcommand\color[2][]{}%
  }%
  \providecommand\transparent[1]{%
    \errmessage{(Inkscape) Transparency is used (non-zero) for the text in Inkscape, but the package 'transparent.sty' is not loaded}%
    \renewcommand\transparent[1]{}%
  }%
  \providecommand\rotatebox[2]{#2}%
  \newcommand*\fsize{\dimexpr\f@size pt\relax}%
  \newcommand*\lineheight[1]{\fontsize{\fsize}{#1\fsize}\selectfont}%
  \ifx\svgwidth\undefined%
    \setlength{\unitlength}{286.26117136bp}%
    \ifx\svgscale\undefined%
      \relax%
    \else%
      \setlength{\unitlength}{\unitlength * \real{\svgscale}}%
    \fi%
  \else%
    \setlength{\unitlength}{\svgwidth}%
  \fi%
  \global\let\svgwidth\undefined%
  \global\let\svgscale\undefined%
  \makeatother%
  \begin{picture}(1,0.3411942)%
    \lineheight{1}%
    \setlength\tabcolsep{0pt}%
    \put(0,0){\includegraphics[width=\unitlength,page=1]{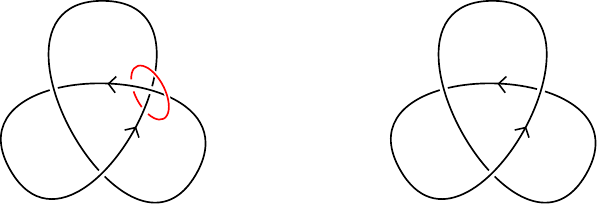}}%
    \put(0.28336075,0.20852625){\color[rgb]{0,0,0}\makebox(0,0)[lt]{\lineheight{0}\smash{\begin{tabular}[t]{l}\color{red}$+1$\end{tabular}}}}%
    \put(0.46815716,0.14214267){\color[rgb]{0,0,0}\makebox(0,0)[lt]{\lineheight{0}\smash{\begin{tabular}[t]{l}\Large$\overset{\partial}{=}$\end{tabular}}}}%
    \put(0.33956473,0.04897279){\color[rgb]{0,0,0}\makebox(0,0)[lt]{\lineheight{0}\smash{\begin{tabular}[t]{l}$\gamma$\end{tabular}}}}%
  \end{picture}%
\endgroup%

\caption{There is a cobordism in $\CP\setminus(\intB \sqcup \intB)$ from LHT to the unknot.}
\label{fig:RHTinCP}
\end{figure}
\end{example}

\begin{example}
\label{ex:t3}
Note that the red $+1$ framed unknot in Figure \ref{fig:RHTinCP} encircles a crossing of LHT, with linking number $0$. If we had instead considered a $-1$ framed unknot encircling a crossing of LHT, with linking number $\pm 2$, then we could argue exactly as in Example \ref{ex:RHT} that there is a cobordism in $\bCP\setminus(\intB \sqcup \intB)$ from LHT to the unknot, hence proving that RHT bounds a disk in $\bCP$. However in this setting the disk is in  homology class $\pm 2H$, where $H$ is a generator of $H_2(\bCP; \Z)$ (the $\pm$ ambiguity is depends on choices of strand orientation and generators of $H_2$, and is inconsequential for this discussion).  After a global change of orientation, we see that LHT bounds a disk in $\CP$ with homology class $\pm 2H$, where $H$ is a generator of $H_2(\CP; \Z)$.
\end{example}

We now give some examples of slice and H-slice knots in the K3 surface. We chose the K3 surface  because it is a symplectic 4-manifold with a rather simple description, but it is not definite, nor is it homeomorphic to a sum of $\CP$, $\bCP$, or $S^2 \times S^2$. The K3 surface can be given a handle decomposition with a 0-handle, twenty-two 2-handles, and a 4-handle, as explained in \cite[Section 8.3]{GS}. See Figure \ref{fig:K3} for such a handle diagram.

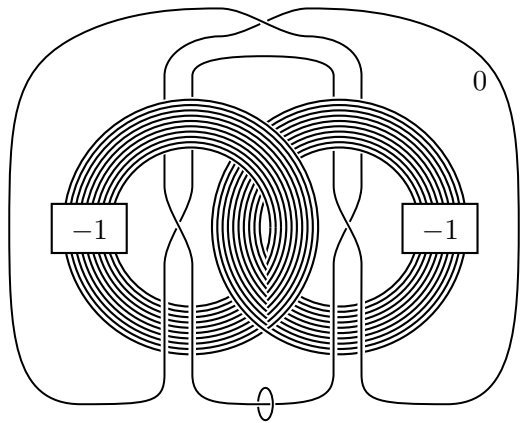
\begin{figure}
\begingroup%
  \makeatletter%
  \providecommand\color[2][]{%
    \errmessage{(Inkscape) Color is used for the text in Inkscape, but the package 'color.sty' is not loaded}%
    \renewcommand\color[2][]{}%
  }%
  \providecommand\transparent[1]{%
    \errmessage{(Inkscape) Transparency is used (non-zero) for the text in Inkscape, but the package 'transparent.sty' is not loaded}%
    \renewcommand\transparent[1]{}%
  }%
  \providecommand\rotatebox[2]{#2}%
  \newcommand*\fsize{\dimexpr\f@size pt\relax}%
  \newcommand*\lineheight[1]{\fontsize{\fsize}{#1\fsize}\selectfont}%
  \ifx\svgwidth\undefined%
    \setlength{\unitlength}{192.7499995bp}%
    \ifx\svgscale\undefined%
      \relax%
    \else%
      \setlength{\unitlength}{\unitlength * \real{\svgscale}}%
    \fi%
  \else%
    \setlength{\unitlength}{\svgwidth}%
  \fi%
  \global\let\svgwidth\undefined%
  \global\let\svgscale\undefined%
  \makeatother%
  \begin{picture}(1,0.80943443)%
    \lineheight{1}%
    \setlength\tabcolsep{0pt}%
    \put(0,0){\includegraphics[width=\unitlength,page=1]{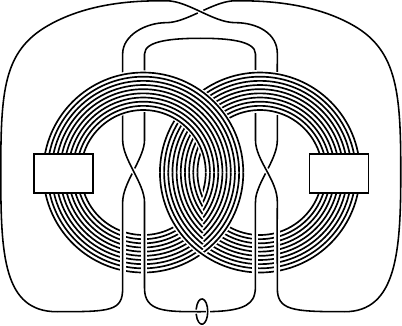}}%
    \put(0.90847269,0.64730961){\color[rgb]{0,0,0}\makebox(0,0)[lt]{\lineheight{0}\smash{\begin{tabular}[t]{l}$0$\end{tabular}}}}%
    \put(0.122732,0.35731197){\color[rgb]{0,0,0}\makebox(0,0)[lt]{\lineheight{0}\smash{\begin{tabular}[t]{l}$-1$\end{tabular}}}}%
    \put(0.80856394,0.35731197){\color[rgb]{0,0,0}\makebox(0,0)[lt]{\lineheight{0}\smash{\begin{tabular}[t]{l}$-1$\end{tabular}}}}%
  \end{picture}%
\endgroup%

\caption{The handle diagram of the K3 surface from \cite[Section 8.3]{GS}. All circles except the $0$-framed right-handed trefoil have framing $-2$.}
\label{fig:K3}
\end{figure}

\begin{example}
\label{ex:LHTinK3}
The left-handed trefoil LHT bounds a disk $\Delta \subset \KT \setminus \intB$ with $[\Delta]^2=0$ and $[\Delta] \neq 0$. To see this, as in Example \ref{ex:RHT}, we locate RHT in $\partial^-$ of the standard handle diagram of $\KT$ with the 0-handle removed. The core of the trefoil-shaped 2-handle in Figure \ref{fig:K3} is a disk in $\KT$ with boundary this RHT. Thus after correcting for outward normal first orientation we see that LHT bounds the desired disk in K3.  There are many ways to see that this disk is nullhomologous, for example after capping it with a Seifert surface one obtains a torus in $K3$ which has non-zero algebraic intersection with a $-2$ framed sphere embedded in $K3$. 
\end{example}




The next lemma implies that the family $\Slice(\KT)$ of knots that are slice in $\KT$ is quite large.
We start by giving a definition.

\begin{definition}
A knot $K_2$ is said to be obtained from a knot $K_1$ by adding a negative twist along $k$ strands if $K_1$ and $K_2$ admit knot diagrams that agree everywhere except in a small region where they appear as shown in Figure \ref{fig:NFT}.
\end{definition}

\begin{figure}
\begingroup%
  \makeatletter%
  \providecommand\color[2][]{%
    \errmessage{(Inkscape) Color is used for the text in Inkscape, but the package 'color.sty' is not loaded}%
    \renewcommand\color[2][]{}%
  }%
  \providecommand\transparent[1]{%
    \errmessage{(Inkscape) Transparency is used (non-zero) for the text in Inkscape, but the package 'transparent.sty' is not loaded}%
    \renewcommand\transparent[1]{}%
  }%
  \providecommand\rotatebox[2]{#2}%
  \newcommand*\fsize{\dimexpr\f@size pt\relax}%
  \newcommand*\lineheight[1]{\fontsize{\fsize}{#1\fsize}\selectfont}%
  \ifx\svgwidth\undefined%
    \setlength{\unitlength}{138.98219682bp}%
    \ifx\svgscale\undefined%
      \relax%
    \else%
      \setlength{\unitlength}{\unitlength * \real{\svgscale}}%
    \fi%
  \else%
    \setlength{\unitlength}{\svgwidth}%
  \fi%
  \global\let\svgwidth\undefined%
  \global\let\svgscale\undefined%
  \makeatother%
  \begin{picture}(1,0.88195071)%
    \lineheight{1}%
    \setlength\tabcolsep{0pt}%
    \put(0,0){\includegraphics[width=\unitlength,page=1]{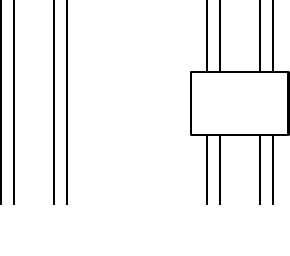}}%
    \put(0.08073113,0.51776234){\color[rgb]{0,0,0}\makebox(0,0)[lt]{\lineheight{1.25}\smash{\begin{tabular}[t]{l}$...$\end{tabular}}}}%
    \put(0.08736321,0.02103424){\color[rgb]{0,0,0}\makebox(0,0)[lt]{\lineheight{1.25}\smash{\begin{tabular}[t]{l}$k$\end{tabular}}}}%
    \put(0.77724886,0.50158611){\color[rgb]{0,0,0}\makebox(0,0)[lt]{\lineheight{1.25}\smash{\begin{tabular}[t]{l}$-1$\end{tabular}}}}%
    \put(0.7930524,0.74441028){\color[rgb]{0,0,0}\makebox(0,0)[lt]{\lineheight{1.25}\smash{\begin{tabular}[t]{l}$...$\end{tabular}}}}%
    \put(0.7930524,0.29111417){\color[rgb]{0,0,0}\makebox(0,0)[lt]{\lineheight{1.25}\smash{\begin{tabular}[t]{l}$...$\end{tabular}}}}%
    \put(0.37213528,0.50280294){\color[rgb]{0,0,0}\makebox(0,0)[lt]{\lineheight{1.25}\smash{\begin{tabular}[t]{l}$\longrightarrow$\end{tabular}}}}%
    \put(0.18482106,0.07569018){\color[rgb]{0,0,0}\rotatebox{90}{\makebox(0,0)[lt]{\lineheight{1.25}\smash{\begin{tabular}[t]{l}\resizebox{19pt}{!}{$\{$}\end{tabular}}}}}%
  \end{picture}%
\endgroup%

\caption{The figure shows the effect of adding a negative full twist along $k$ strands. The orientations on the strands are arbitrary.}
\label{fig:NFT}
\end{figure}

\begin{lemma}
\label{lem:untwistingK3}
Let $K_0$, $K_1$, and $K_2$ be knots such that $K_{i}$ is obtained from $K_{i+1}$ by adding a negative full twist along $k_i$ strands, and suppose that $k_i \leq 5$. Then, there exists a smooth, connected, properly embedded surface $\Sigma \subset \KT\setminus\intB$ with $\de\Sigma = K_2$ and
\[
g(\Sigma) = g_4(K_0).
\]
\end{lemma}

\begin{proof}
Let $X$ be a closed four-manifold with a handle diagram with no 1-handles, and let $W = X \setminus (\intB \sqcup \intB)$, seen as a cobordism from $\del^- W=S^3$ to $\del^+ W=S^3$. We think of $W$ as removing a 0-handle and a 4-handle from $X$.

We first observe that if any $T_{k,-k}$ torus link appears as a sublink of such a handle diagram (where we make no assumption on the framings of the components), and the knot $J$ is obtained from $K$ by adding a negative twist along $k$ strands, then there is an embedded annulus in $W$ from $J \subset \de^- W$ to $K \subset \de^+ W$. Such an annulus is shown in Figure \ref{fig:slideK3}.

Now observe that the handle diagram of $\KT$ in Figure \ref{fig:K3} contains two split copies of $T_{5,-5}$, namely the 5 outermost components of the left bundle and the 5 innermost components of the right bundle. Thus, there is an embedded annulus $A$ in $W = \KT \setminus (\intB \sqcup \intB)$ from $K_0$ to $K_2$. To get the surface $\Sigma$, we attach a minimal genus surface for $K_0$ in a $B^4$ glued to $\de^- W$. 
\end{proof}

\begin{figure}
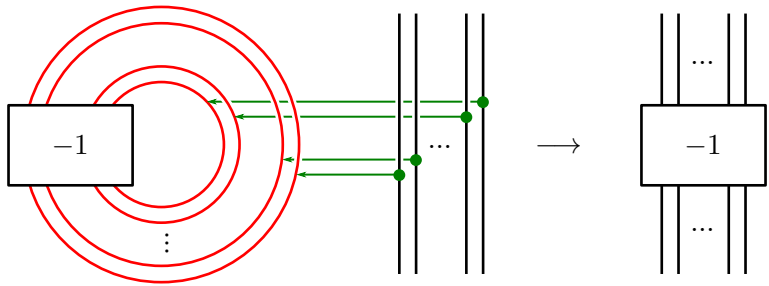
\caption{For a 4-manifold $W$ as in the proof of Lemma \ref{lem:untwistingK3}, the red curves on the left represent some 2-handles of $W$ arranged as a $T_{k,-k}$ torus link. If we slide $k$ parallel strands of a knot $K\subset \de^+ W$ over the 2-handles as shown in the figure, the resulting knot $J\subset \de^- W$ is obtained from $K$ by adding a negative full twist along $k$ strands.}
\label{fig:slideK3}
\end{figure}

\begin{corollary}\label{cor:unknotting}
Any knot $K$ with unknotting number $u(K) \leq 2$ is slice in $\KT$.
\end{corollary}
\begin{proof}
Suppose that the knots $K_0 = U$ (the unknot), $K_1$, and $K_2=K$ are such that $K_{i}$ is obtained from $K_{i+1}$ by a crossing change.
Note that any crossing change can be realized by adding a negative full twist along 2 strands. By applying Lemma \ref{lem:untwistingK3}, we obtain a slice disk $\Sigma$ for $K$.
\end{proof}

\begin{example}
From Corollary \ref{cor:unknotting} we see that both trefoils are slice in K3. (See also Example \ref{ex:LHTinK3}.) However, note that neither trefoil is H-slice in K3; see Example~\ref{ex:torus} below.
\end{example}

The following lemma illustrates a way to construct H-slice knots in a general 4-manifold.

\begin{lemma}
\label{lem:Wh}
Let $X$ be a 4-manifold and let $K$ be a knot in $S^3$ which bounds a surface $\Sigma$ of genus $g$  in $\Xo$.
Then $\Wh^{\pm}_{-[\Sigma]^2}(K)$ bounds a homologically trivial embedded surface $\Wh^{\pm}_{-[\Sigma]^2}(\Sigma)$ of genus $2g$ in $\Xo$.
\end{lemma}
\begin{proof}
Take two copies $\Sigma_1$ and $\Sigma_2$ of $\Sigma$ with opposite orientations, with boundary the untwisted cable $C_{2,0}(K)$, where the two components have opposite orientations. Note that $\Sigma_1 \cdot \Sigma_2 = -[\Sigma]^2$.
We can remove these intersections by introducing twists into the cable; thus, we get a 2-component embedded surface with boundary $C_{2, -2[\Sigma]^2}(K)$, i.e.~so that the linking number of the two components is equal to $[\Sigma]^2$. By adding a twisted band to connect the two components we get a null-homologous embedded surface $\Wh^\pm_{-[\Sigma]^2}(\Sigma)$ with boundary $\Wh^\pm_{-[\Sigma]^2}(K)$.
\end{proof}

\begin{example}
By Lemma \ref{lem:Wh} applied to Example \ref{ex:LHTinK3}, we see that $\Wh^{\pm}_0($LHT) is H-slice in K3.
\end{example}

\section{Topological obstructions}
\label{sec:top}
We review here some constraints on the homology classes of surfaces embedded in topological four-manifolds.

\subsection{The Arf invariant}
The Kirby-Siebenmann invariant $\ks(X) \in H^4(X; \Z/2)$ is an obstruction to smoothing topological manifolds; see \cite[p. 252]{KS}. When $X$ is a closed oriented topological $4$-manifold, then $\ks(X)$ is valued in $\Z/2$ and given by the formula 
\begin{equation}
\label{eq:ks}
\ks(X) = \frac{1}{8}(\sigma(X) - [\Sigma]^2) - \Arf(X, \Sigma) \pmod {2},
\end{equation}
where $\Sigma$ is any characteristic surface in $X$; see \cite[Corollary 9.3]{KirbyTaylor}. (Here, the Arf invariant of $(X, \Sigma)$ is a quadratic enhancement of the intersection form on $H_1(\Sigma; \Z/2)$; see  \cite[p.120-121]{Matsumoto}.) 

In particular, when $X$ is smooth, we have $\ks(X)=0$ and we obtain:
\begin{equation}
\label{eq:Arf} \frac{\sigma(X) - [\Sigma]^2}{8} \equiv \Arf(X, \Sigma) \pmod{2}.
\end{equation}
This result is due to Rokhlin \cite{RokhlinGudkov}; see \cite{FK}, \cite{Matsumoto} for different proofs. 

There is also a relative version of \eqref{eq:ks}, as follows.

\begin{theorem}
\label{thm:FK}
Let $X$ be a  closed, connected, topological $4$-manifold. If $\Sigma \subset \Xo$ is a properly embedded, locally flat characteristic surface with boundary a knot $K$, then
\begin{equation}
\label{eq:ArfSigma}
  \frac{\sigma(X) - [\Sigma]^2}{8} \equiv \Arf(K) + \Arf(X, \Sigma) + \ks(X) \pmod{2}.
  \end{equation}
  \end{theorem}
In the smooth case, this appears (with minor modifications) as Corollary 6 on \cite[p.~69]{KirbyBook}, and also as Theorem 2.2 in \cite{Yasuhara3}. The topological case can be deduced from the smooth one, by gluing a Seifert surface for $K$ to $\Sigma$ and applying Equation~\eqref{eq:ks} to the resulting closed surface in $X = \Xo \cup B^4$. Theorem~\ref{thm:FK} is also a particular case of a result of Klug  \cite[Theorem 4]{Klug}, who studied the more general setting of $4$-manifolds with boundary a homology sphere.

When $X$ is smooth, spin, and $\Sigma$ is a null-homologous disk, we have $\sigma(X) \equiv 0\!\! \pmod{16}$ by Rokhlin's theorem. We also have $\Arf(X, \Sigma)=0$, because $H_1(\Sigma; \Z/2)=0$. From \eqref{eq:ArfSigma}, we recover an old result of Robertello~\cite[p.~1-2]{robertello}:
 \begin{theorem}[Robertello \cite{robertello}] 
\label{thm:Arf}
If a knot $K$ is topologically H-slice in a spin smooth $4$-manifold, then $\Arf(K)=0$.
\end{theorem}

\begin{example}
\label{ex:torus}
Recall that the Arf invariant of a knot $K$ can be read from the determinant $D(K) = |\Delta_K(-1)|$: we have $\Arf(K)=0 \iff D(K) \equiv \pm 1 \ \ \pmod{8}.$ For example, the torus knot $T_{2, 2k+1}$ has $D(T_{2, 2k+1}) = 2k+1$ and therefore $\Arf(T_{2, 2k+1})=0 \iff k \equiv 0$ or $3 \pmod{4}$. Thus, torus knots of the form $T_{2, 8k+3}$ and $T_{2,8k+5}$ are not H-slice in smooth spin $4$-manifolds.
\end{example}

\subsection{Levine-Tristram signatures} Given a knot $K \subset S^3$ and a value $\omega \in S^1$, the Levine-Tristram signature $\sigma_K(\omega)$ is defined as the signature of $(1-\omega)A + (1 -\overline{\omega}) A^T$, where $A$ is a Seifert matrix for $K$; see \cite{Tristram}, \cite{Levine}, or \cite{ConwaySurvey}.

Following \cite{NP}, we denote by $\Se$ the set of unit complex numbers that are not zeros of any integral Laurent polynomial $p$ with $p(1)=1$. Note that $\Se$ includes, for example, roots of unity of order a prime power. The evaluations $\sigma_K(\omega)$ for $\omega \in \Se$ are knot concordance invariants; see \cite{NP}. In particular, for $\omega=-1$, we obtain the usual knot signature $$\sigma(K) :=\sigma_K(-1).$$

The following result is a special case of Theorem 3.8 in \cite{ConwayNagel}. 

\begin{theorem}[Conway-Nagel \cite{ConwayNagel}]
\label{thm:CN}
Let $X$ be a closed topological $4$-manifold with $H_1(X; \Z)=0$ and signature $\sigma(X)$. If the knot $K \subset S^3$ bounds a locally flat, properly embedded, null-homologous surface in $\Xo$ of genus $g$, then 
\begin{equation}
\label{eq:CN}
 |\sigma_K(\omega) + \sigma(X)|  \leq b_2(X) + 2g,
 \end{equation}
for every $\omega \in \Se$. 
\end{theorem}

Note that we can re-write the constraint \eqref{eq:CN} as
$$\sigma_K(\omega) \in [-2b_2^+(X)-2g, 2b_2^-(X)+2g].$$

One case of interest is when $g=0$: 
\begin{corollary}[Corollary 4.7 in \cite{KlugRuppik}]
\label{cor:KR}
Let $X$ be a closed topological $4$-manifold with $H_1(X; \Z)=0$. If a knot $K$ is topologically H-slice in $X$, and $\omega \in \Se$, then $\sigma_K(\omega) \in [-2b_2^+(X), 2b_2^-(X)]$. 
\end{corollary}
 As an application of this, Klug and Ruppik showed in \cite[Theorem 4.4]{KlugRuppik} that for every closed oriented topological $4$-manifold $X$, there exists a knot $K \subset S^3$ that is not topologically H-slice in $X$.

\subsection{Relative Rokhlin-type inequalities} 
Rokhlin \cite{Rokhlin} gave constraints on the homology classes of closed embedded surfaces inside $4$-manifolds, provided these homology classes are divisible by a prime. (Note that this includes the null-homologous case.) Rokhlin's results were used by Yasuhara in \cite{YasuharaExample, yasuhara} to prove the existence of non-slice knots in $\CP$. Rokhlin's results were also used by Klug and Ruppik in \cite{KlugRuppik} to show that every closed $2$-handlebody $X \neq S^4$ admits knots that are slice in $X$ but not slice in $S^4$.

Rokhlin's constraints were generalized to surfaces with boundary a knot $K$ in the work of Viro \cite{Viro2} and Gilmer \cite[Theorem 4.1 and Remark (a) on p.371]{Gilmer4}. Note that while Rokhlin, Viro and Gilmer all worked in the smooth category, the main ingredient in their proofs is the $G$-signature theorem \cite{AS}, which also works in the topological category when applied to tame, semi-free actions (such as, in our case, the action of deck transformations on a cyclic branched cover); see \cite[Theorem 14B.2]{Surgery}. Consequently, we will phrase the relative result in the topological category.

\begin{theorem}[Viro \cite{Viro2}, Gilmer \cite{Gilmer4}]
\label{thm:rokhlin}
Let $X$ be a topological closed oriented $4$-manifold with $H_1(X; \Z)=0$ and with signature $\sigma(X)$. Let $\Sigma \subset \Xo$ be a locally flat, properly embedded surface of genus $g$, with boundary a knot $K\subset S^3$. If the homology class $[\Sigma] \in H_2(\Xo, \del \Xo) \cong H_2(X)$ is divisible by a prime power $m=p^k$, then 
\begin{equation}
\label{eq:Gsignature}
\left |\sigma_K(e^{2\pi r i/m}) + \sigma(X) - \frac{2r(m-r)\cdot [\Sigma]^2}{m^2} \right| \leq b_2(X)+2g,
\end{equation}
for every $r =1, \dots, m-1.$
\end{theorem}

\begin{remark}
When $[\Sigma]$ is $2$-divisible, taking $m=2$ in Theorem~\ref{thm:rokhlin} we get:
\begin{equation}
\label{eq:signature}
\left | \sigma(K)+ \sigma(X)   - \frac{ [\Sigma]^2}{2}\right| \leq b_2(X)+2g.
\end{equation}
\end{remark}

\begin{remark}
Suppose $K$ is the unknot. After capping $K$ with a disk in $B^4$, the inequality \eqref{eq:Gsignature} for $r = \lfloor m/2 \rfloor$ gives Rokhlin's main theorem from \cite{Rokhlin}. 
\end{remark}

\begin{remark}
When $\Sigma$ is null-homologous, Theorem~\ref{thm:rokhlin} reduces to the Conway-Nagel result (Theorem~\ref{thm:CN}) for $\omega=e^{2\pi r i/m}$.
\end{remark}

\section{Adjunction inequalities}
\label{sec:adj}

\subsection{The adjunction inequality for closed surfaces in closed $4$-manifolds}

We start by reviewing the adjunction inequality from Seiberg-Witten theory, which gives genus bounds on smoothly embedded surfaces in $4$-manifolds.

Let $X$ be a closed $4$-manifold with $b_2^+(X)>1$. 
Recall that  $X$ is called {\em of Seiberg-Witten simple type} if the Seiberg-Witten invariants $SW_{X, \s}$ vanish whenever the expected dimension of the Seiberg-Witten moduli space,
$$ d(\s) = \frac{c_1(\s)^2 - (2\chi(X) +3 \sigma(X))}{4},$$
is nonzero. It is known that complex projective surfaces and, more generally, symplectic $4$-manifolds are of simple type; cf.~\cite{Taubes1}, \cite{Taubes2}.

\begin{theorem}[\cite{KM-thom}, \cite{MST}, \cite{os-thom}]
\label{thm:SW-adjunction}
Let $X$ be a closed $4$-manifold with $b_2^+(X)>1$. Let $\Sigma \subset X$ be a smoothly embedded surface of genus $g(\Sigma) > 0$. Suppose that either
\begin{enumerate}[(a)]
\item $[\Sigma]^2 \geq 0$; or
\item $X$ is of Seiberg-Witten simple type.
\end{enumerate}
Then, for each \spinc structure $\s \in \Spinc(X)$ for which $SW_{X, \s} \neq 0$, we have
\begin{equation}
\label{eq:SW-adjunction}
\scal{c_1(\s)}{[\Sigma]} + [\Sigma]^2 \leq 2g(\Sigma)-2.
\end{equation}
\end{theorem}

In Heegaard Floer theory, the analogues of the Seiberg-Witten invariants are the Ozsv\'ath-Szab\'o mixed invariants $\Phi_{X, \s}$ defined in \cite{os4}. The invariants $\Phi_{X, \s}$ are conjecturally equal to the Seiberg-Witten invariants, and have similar properties. We recall their definition in Section \ref{sec:adj-bdy}.

We say that $X$ is {\em of Ozsv\'ath-Szab\'o simple type} if $\Phi_{X,\s}=0$ whenever $d(\s) \neq 0$. 
 It is expected that symplectic 4-manifolds are of Ozsv\'ath-Szab\'o simple type; see ~\cite[Conjecture 1.3]{JabukaMark}. In any case, we know that the K3 surface is of Ozsv\'ath-Szab\'o simple type by the calculation in \cite[Section 4]{os-symp}; see \cite{JabukaMark} for other examples.

\begin{theorem}[Ozsv\'ath-Szab\'o \cite{os4, os-symp}]
\label{thm:OSz-adjunction}
Let $X$ be a closed $4$-manifold with $b_2^+(X)>1$. Let $\Sigma \subset X$ be a smoothly embedded surface of genus $g(\Sigma) > 0$. Suppose that either
\begin{enumerate}[(a)]
\item $[\Sigma]^2 \geq 0$; or
\item $X$ is of Ozsv\'ath-Szab\'o simple type.
\end{enumerate}
Then, for each \spinc structure $\s \in \Spinc(X)$ for which $\Phi_{X, \s} \neq 0$, we have
\begin{equation}
\label{eq:GS-adjunction}
\scal{c_1(\s)}{[\Sigma]} + [\Sigma]^2 \leq 2g(\Sigma)-2.
\end{equation}
\end{theorem}

Under the hypothesis (a), this result is \cite[Theorem 1.5]{os4}. An alternative proof was given by Zemke \cite[Theorem 1.6]{ZemkeAdj}. Under the hypotheses (b), it is a consequence of the adjunction relation in \cite[Theorem 3.1]{os-symp}. It is also a particular case of Theorem~\ref{thm:OSz-adjunction-bdy}, which we will prove below.

In Seiberg-Witten theory, there is also an adjunction inequality for embedded spheres:
\begin{theorem}[Fintushel-Stern \cite{FSimmersed}]
\label{thm:FS}
Let $X$ be a closed $4$-manifold with $b_2^+(X)>1$. Suppose that there exists a \spinc structure $\s$ with $SW_{X, \s} \neq 0$. Then, there exist no smoothly embedded spheres $\Sigma \subset X$ such that $[\Sigma]^2 \geq 0$ and $[\Sigma]\neq 0$.
\end{theorem}

\begin{remark}
The proof of Theorem~\ref{thm:FS} involves the blow-up formula and the finiteness of basic classes. Since the same results also hold in the Heegaard Floer setting (cf.~\cite[Theorems 1.4 and 3.3]{os4}), the adjunction inequality for spheres also holds if we replace the condition $SW_{X, \s} \neq 0$ with $\Phi_{X, \s} \neq 0$.
\end{remark}

For future reference, we will need a refinement of Theorem~\ref{thm:FS}, based on the Bauer-Furuta invariants instead of the Seiberg-Witten invariants. The Bauer-Furuta invariant takes values in an (equivariant) stable homotopy group of spheres, and is defined from the Seiberg-Witten map by using finite dimensional approximation. For simplicity, we will only consider the nonequivariant Bauer-Furuta invariant, with values in the ordinary stable homotopy groups of spheres:
$$\BF_{X, \s} \in \pi_{d(\s)+1}^{\st}(S^0).$$
Intuitively, via the Pontryagin-Thom construction, this captures the framed cobordism class of the Seiberg-Witten moduli space. Unlike the Seiberg-Witten invariant, it can be nonzero even when $b_2^+(X)$ is even; see \cite{BF}, \cite{Bauer} for more details.

The following theorem was independently proved by Yasui \cite[Theorem 2.8]{Yasui} using Fr{\o}yshov's work \cite[Theorem 1.1]{Froyshov}, and by Khandhawit-Lin-Sasahira in \cite[Corollary 1.9]{KLS}.
\begin{theorem}[\cite{Yasui}, \cite{KLS}]
\label{thm:BFadj}
Let $X$ be a closed $4$-manifold with $b_2^+(X)>1$. Suppose that there exists a \spinc structure with $\BF_{X, \s} \neq 0$. Then, there exist no smoothly embedded spheres $\Sigma \subset X$ such that $[\Sigma]^2 \geq 0$ and $[\Sigma]\neq 0$.
\end{theorem}

\subsection{The adjunction inequality for closed surfaces in 4-manifolds with boundary}
\label{sec:adj-bdy}
An inequality of this form was proved by Zemke \cite[Theorem 1.6]{ZemkeAdj}, assuming the map on $\mathit{CFL}^-$ is nonzero. Here we will prove a similar inequality involving mixed invariants.

The mixed invariant for a 4-manifold $X$ with boundary $Y$ and $b_2^+(X)>1$ is defined in \cite[Section 8]{os4} as follows: We choose an admissible cut $N$ that splits $X$ as $X_1 \cup_N X_2$, with $\del X_1 = N$ and $\del X_2 = \overline{N} \sqcup Y.$ For a \spinc structure $\s \in \Spinc(X)$, we combine the minus map on $X_1$ with the plus map on $X_2$ and the identification $\HFp_{\red}(N) \cong \HFm_{\red}(N)$ in the middle, to obtain a mixed map
\[
\HFm(S^3) \to \HFp(Y, \s|_Y).
\]
By incorporating the action of $\A(X) = \Lambda^*(H_1(X;\Z)/\Tors)[U]$, we get a map
\[
\Phi_{X,\s} \colon \A(X) \otimes_{\Z[U]} \HFm(S^3) \to \HFp(Y, \s|_Y),
\]
which we call the \emph{Ozsv\'ath-Szab\'o mixed invariant}.
We remark that we consider all the Heegaard Floer modules over $\F_2[U]$, where $\F_2$ is the field with two elements (see \cite{JTZ} for a discussion of the coefficient ring).

For a cobordism $W$ between non-empty manifolds $Y_0$ and $Y_1$, together with a \spinc structure $\s$ which is torsion on $Y_0$ and $Y_1$, we define the quantity
\[
\dsh(W, \s) = \frac{c_1^2(\s) - (2\chi(W)+3\sigma(W))}4,
\]
which is additive under composition of \spinc cobordisms. 

Note that for a closed 4-manifold $X$ with a \spinc structure $\s$, we have $d(\s) = \dsh(W, \s) -1$, where $W$ is $X\setminus(\intB \sqcup \intB)$ seen as a cobordism from $S^3$ to $S^3$. Analogously, for a 4-manifold $X$ with one boundary component $Y$ and a \spinc structure $\s$ which is torsion on $Y$, we define
\[
d(\s) = \dsh(W, \s|_W) - 1 = \frac{c_1(\s)^2 - (2\chi(X) + 3\sigma(X))}4 - \frac12,
\]
where $W = X \setminus \intB$ is seen as a cobordism from $S^3$ to $Y$.

\begin{definition}
\label{def:relsimple}
Let $X$ be a smooth 4-manifold $X$ with $\de X = Y$ and $b_2^+(X)>1$.
We say that $X$ is \emph{of relative Ozsv\'ath-Szab\'o simple type} if $Y$ is a rational homology sphere and $\Phi_{X,\s}=0$ whenever $d(\s) \neq d(Y, \s|_Y)$. 
\end{definition}

Note that when $X$ is closed, then $X$ being of Ozsv\'ath-Szab\'o simple type (in the usual sense) is equivalent to $\Xo = X \setminus \intB$ being of relative Ozsv\'ath-Szab\'o simple type as in Definition~\ref{def:relsimple}.

\begin{theorem}
\label{thm:OSz-adjunction-bdy}
Let $X$ be a smooth $4$-manifold (possibly with boundary) and $b_2^+(X)>1$. Let $\Sigma \subset \Int(X)$ be a smoothly embedded closed connected surface of genus $g(\Sigma) > 0$. Suppose that either

\begin{enumerate}[(a)]
\item $[\Sigma]^2 \geq 0$; or
\item $X$ is of relative Ozsv\'ath-Szab\'o simple type.
\end{enumerate}
Then, for each \spinc structure $\s \in \Spinc(X)$ for which $\Phi_{X, \s} \neq 0$, we have
\begin{equation}
\label{eq:OSz-adjunction-bdy}
\scal{c_1(\s)}{[\Sigma]} + [\Sigma]^2 \leq 2g(\Sigma)-2.
\end{equation}
\end{theorem}

\begin{proof}
In case (a) we follow Ozsv\'ath-Szab\'o's proof of \cite[Theorem 1.5]{os4}. For any splitting of $X$ into two parts $X_1$ and $X_2$ along an admissible cut, we can reduce to the case $[\Sigma]^2 = 0$ by repeatedly blowing up: each copy of $\bCP$ can be added to either $X_1$ or $X_2$, and by \cite[Theorem 3.7 with $\ell=0$]{os4} the mixed invariant map is still non-vanishing. The surface $\Sigma$ is replaced by the connected sum of itself with the exceptional divisor, and the \spinc structure on $\bCP$ is chosen so that $\scal{c_1(\s)}{[\Sigma]} + [\Sigma]^2$ does not change.

We now assume $[\Sigma]^2=0$. 
Since $b_2^+(X)>1$, there exists a homology class $\alpha \in H_2(X;\Z)$ with $\alpha^2>0$ and $\alpha\cdot [\Sigma] =0$.
We represent $\alpha$ by a smoothly embedded surface $T$ transverse to $\Sigma$. By adding tubes on $T$ to cancel intersections with $\Sigma$, we can also assume that $T$ and $\Sigma$ are disjoint (note that $[T]=\alpha$ does not change while doing so). Let $\nbd(T)$ be a tubular neighborhood of $T$. 
Then $N = \de\nbd(T)$ is a new admissible cut of $X$, and we can assume that $\Sigma \subset X_2$.
Since $[\Sigma]^2 = 0$, the boundary of a neighborhood of $\Sigma$ is $S^1 \times \Sigma$. By tubing this $S^1 \times \Sigma$ to the boundary of $X$, we find that the cobordism map
\[
F_{X_2, \s}^+ \colon \HFp(N, \s) \to \HFp(\del X, \s)
\]
factors through $\HFp(\del X \# (S^1 \times \Sigma), \s)$. This group vanishes by \cite[Theorem 7.1]{OSz-properties}, unless
\[
|\scal{c_1(\s)}{[\Sigma]}| \leq 2g(\Sigma)-2.
\]
This concludes the proof of case (a), since the vanishing of $\HFp(\del X \# (S^1 \times \Sigma), \s)$ implies that of $\Phi_{X,\s}$ too.

We now turn to case (b).
Assume by contradiction that the inequality \eqref{eq:OSz-adjunction-bdy} does not hold.
By blowing up appropriately we can repeatedly reduce the quantity $\scal{c_1(\s)}{[\Sigma]} + [\Sigma]^2$ by $2$ until
\[
\scal{c_1(\s)}{[\Sigma]} + [\Sigma]^2 = 2g(\Sigma),
\]
while keeping $\Phi_{X, \s} \neq 0$. (Recall that the left hand side of the previous equation is always even because $c_1(\s)$ is characteristic.)

As before, we find an admissible cut of $X$ with $\Sigma \subset X_2$.
Then, by \cite[Theorem 1.5]{ZemkeAdj} applied to the orientation reversal $\overline{\Sigma}$, with $\Sigma_\w$ being a disk and $\Sigma_\z = \overline{\Sigma} \setminus \Sigma_\w$, we have
\[
F_{X_2, \s}^+(-) = F_{X_2, \s+PD([\Sigma])}^+(\iota_*(\xi(\Sigma_\z)) \otimes -),\]
following the notation of \cite{ZemkeAdj}. Thus, on the mixed invariant level, we get that
\[
\Phi_{X, \s+PD([\Sigma])}(\iota_*(\xi(\Sigma_\z)) \otimes -) = \Phi_{X, \s} (-) \neq 0,
\]
showing that $\Phi_{X, \s+PD([\Sigma])} \neq 0$.

Note that $\s$ and $\s + PD([\Sigma])$ restrict to the same \spinc structure on $\de X$: this is because the condition $\Sigma \subset \Int(X)$ implies that $PD([\Sigma]) \in H^2(X)$ maps to $0 \in H^2(\de X)$ under the map induced by restriction, hence $PD([\Sigma])$ acts trivially on $\Spinc(\de X)$.

By arguing as in \cite[Corollary 1.7]{os-thom}, we get  
\begin{align*}
d(\s+PD([\Sigma])) &= \frac{(c_1(\s)^2 + 4\scal{c_1(\s)}{[\Sigma]} + 4[\Sigma]^2) - (2\chi(X)+3\sigma(X))}4  -\frac12 \\
&= d(\s) + \scal{c_1(\s)}{[\Sigma]} + [\Sigma]^2 = d(\s) + 2g(\Sigma) >d(\s).
\end{align*}
Both $\Phi_{X, \s}$ and $\Phi_{X, \s+PD([\Sigma])}$ are nonzero, but at least one of $d(\s)$ and $d(\s+PD([\Sigma]))$ must be different from $d(Y, \s|_Y) = d(Y, (\s + PD([\Sigma]))|_Y).$ This is a contradiction with the relative Ozsv\'ath-Szab\'o simple type assumption.
\end{proof}

\subsection{Relative adjunction inequalities}
We will be interested in surfaces with boundary in $4$-manifolds of the form $\Xo=X \setminus \intB$, where $X$ is closed. The adjunction inequality for closed surfaces (Theorem \ref{thm:OSz-adjunction}) has the following immediate consequence.

\begin{theorem}
\label{thm:adjunction-g4}
Let $\Sigma \subset \Xo$ be a properly embedded surface with $g(\Sigma) + g_4(K) >0$, where $X$ is a closed 4-manifold with $b_2^+(X)>1$. 
Suppose that either
\begin{enumerate}[(a)]
\item $[\Sigma]^2 \geq 0$; or
\item $X$ is of Ozsv\'ath-Szab\'o simple type.
\end{enumerate}
Then, for each \spinc structure $\s \in \Spinc(X)$ for which $\Phi_{X, \s} \neq 0$, we have
\begin{equation}
\label{eq:adj-g4}
\scal{c_1(\s)}{[\Sigma]} + [\Sigma]^2 \leq 2g(\Sigma)-2 + 2g_4(K).
\end{equation}
\end{theorem}

\begin{proof}
Choose a surface $S \subset B^4$ with boundary $\mK$ and genus $g_4(K)$. By gluing $\Sigma$ and $S$ together we get a closed surface in $X$, to which we can apply Theorem \ref{thm:OSz-adjunction}.
\end{proof}

We will refine this using the invariant $\nu^+$ constructed by Hom and Wu in \cite{HW}.
Let $\CFKi(K)$ denote the ($\Z \oplus \Z$)-filtered knot Floer complex of a knot $K$ in $S^3$, and let $i$ and $j$ denote the two filtration indices. Following \cite{HomSurvey}, we define $A_s^- = \mc C\set{\max\set{j-s,i}\leq0}$ and $B^- = \mc C\set{i\leq0}$, where $\mc C\{ \phantom{\cdot} \}$ denotes the subcomplex of $\CFKi(K)$ given by the respective inequality. The inclusion map $v_s^- \colon A_s^- \to B^-$ induces a map $v_{s,*}^-$ in homology, and we define
\[
V_s(K) := \rk_{\F_2} (\coker v_{s,*}^-).
\]
The numbers $V_s(K)$ are concordance invariants. They are non-negative, non-increasing in $s$, and they vanish for $s\gg0$.
They were originally defined by Rasmussen~\cite{Rasmussen}, although the notation $V_s(K)$ was introduced in \cite[Section 8]{RatSurg}.
The invariant $\nu^+$ is defined as
\[
\nu^+(K) := \min\set{s \geq 0 \,\middle|\, V_s(K) = 0}.
\]
(Strictly speaking, this is the definition of $\nu^-(K)$, but $\nu^-(K) = \nu^+(K)$ by \cite[Proposition 2.13]{Upsilon}.)

{
\renewcommand{\thethm}{\ref{thm:adjunction-nu+}}
\begin{theorem}
Let $X$ be a closed 4-manifold, with $b_2^+(X)>1$. Let $\Sigma \subset \Xo$ be a properly embedded surface with $g(\Sigma) > 0$ and $\del \Sigma = K$, and let $\mK$ denote the mirror of $K$. Suppose that either  $[\Sigma]^2 \geq 2\nu^+(\mK)$ or $X$ is of Ozsv\'ath-Szab\'o simple type. Then, for every \spinc structure $\s \in \Spinc(X)$ for which the mixed invariant $\Phi_{X, \s}$ is non-zero, we have
\begin{equation}
\label{eq:adjunction-nu+}
\scal{c_1(\s)}{[\Sigma]} + [\Sigma]^2 \leq 2g(\Sigma)-2 + 2\nu^+(\mK).
\end{equation}
\end{theorem}
\addtocounter{thm}{-1}
}

\begin{proof}
We begin by setting up some objects and notation: Let $n \geq 0$. We denote by $\hat X_{-n}$ the manifold obtained by attaching a $(-n)$-framed 2-handle along $K$ to $\Xo$. We denote the trace of the surgery by $X_{-n}(K)$, so that $\hat X_{-n} = \Xo \cup_{S^3} X_{-n}(K)^\circ$.
Let $\hat\Sigma_{-n}$ denote the surface obtained by capping off $\Sigma$ with the core of the 2-handle. 

Let $\hat\s$ be a \spinc structure on $\hat X_{-n}$. Consider the punctured \emph{positive} trace $X_n(\mK)^\circ$. Let $H \in H_2(X_n(\mK)^\circ)$ be the homology class given by the core of the 2-handle union the cone on the knot $K$. For $i \in \Z$, let $\s_i$ denote the spin$^c$ structure on  $X_n(\mK)^\circ$ with $\scal{c_1(\s_i)}{H} = n + 2i$. We denote by $\t_i$ its restriction to $S^3_n(\mK)$; note that $\t_i = \t_{n+i}$. Following usual conventions, we let $H$ and $\s_i$ denote also the corresponding homology class and spin$^c$ structure on the orientation reversal $X_{-n}(K)^\circ$, which has $\scal{c_1(\s_i)}{H} = -n-2i$. 

Notation in hand, we observe that Theorem \ref{thm:OSz-adjunction-bdy} gives
\begin{equation}
\label{eq:adj-nu+-pf-1}
\scal{c_1(\hat\s)}{[\hat \Sigma_{-n}]} + [\hat \Sigma_{-n}]^2 \leq 2g(\hat\Sigma_{-n})-2
\end{equation}
whenever we are in the setting
\begin{itemize}
\item $[\hat \Sigma_{-n}]^2 =[\Sigma]^2 - n \geq 0$ or $\hat X_{-n}$ is of Ozsv\'ath-Szab\'o simple type;
\item $g(\hat\Sigma_{-n}) = g(\Sigma)>0$; and
\item the map induced by the 2-handle cobordism $X_{-n}(K)^\circ$,
\begin{equation}
\label{eq:HFp}
F^+ \colon \HFp(S^3) \to \HFp(S^3_{-n}(K), \hat\s|_{S^3_{-n}(K)}),
\end{equation}
is nonzero on the bottom element. 
\end{itemize}

We now check that we can choose $n$ and $\hat \s$ so that these conditions are satisfied. We will first assume that either $[\Sigma]^2>0$ or $X$ is of Ozsv\'ath-Szab\'o simple type. We will deal with the special case $[\Sigma]^2=0$ at the end. Notice that the second bullet point is satisfied by hypothesis and when $X$ is not of Ozsv\'ath-Szab\'o simple type we can satisfy the first half of the first bullet point by  choosing $0<n\le[\Sigma]^2$.

When $X$ is of Ozsv\'ath-Szab\'o simple type, we will now argue that $\hat X_{-n}$ is of relative Ozsv\'ath-Szab\'o simple type for all $n>0$. Suppose that $\Phi_{\hat X_{-n}, \hat \s} \neq 0$ for some $\hat \s =\s \# \s_i$. We choose an admissible cut for $X$ and use the same cut for the manifold $\hat X_{-n}$. Then one can immediately see that 
$$\Phi_{\hat X_{-n}, \hat \s} = F^+ \circ \Phi_{X, \s},$$ and therefore $\Phi_{X, \s} \neq 0$. Since $d(\s) = 0$ (by the simple type assumption) and since the action of $\Lambda^k(H_1(X;\Z)/\Tors)$ lowers the degree by $k$, by grading considerations we have
\[
\Phi_{X, \s}(h \otimes \xi) = 0
\]
whenever $h \in \bigoplus_{k \geq 1} \Lambda^k(H_1(X;\Z)/\Tors) \subset \A(X)$ or $\xi \in U \cdot \HFm(S^3) \subset \HFm(S^3)$. Thus, the fact that $\Phi_{X, \s} \neq 0$ implies that the following relation holds:
\[
\Phi_{X, \s}(1_{\A(X)} \otimes 1_{\HFm(S^3)}) = 1_{\HFp(S^3)}.
\]
Thus, the fact that $\Phi_{\hat X_{-n}, \hat\s} \neq 0$ implies that $F^+(1_{\HFp(S^3)}) \neq 0$.
Because the map induced by $X_{-n}(K)^\circ$ on $\HFi$ is nontrivial for $n > 0$ and for all \spinc structures~\cite[Proposition 9.4]{OSzAbsolutely}, $F^+$ must map $1_{\HFp(S^3)}$ nontrivially to the tower of $\HFp(Y, \t_i)$.
Then, the $U$-equivariance forces $F^+$ to send $1_{\HFp(S^3)}$ to the bottom element $\x$ of such a tower. Thus, we can compute
\[
\gr(F^+) = \gr(\x) - \gr(1_{\HFp(S^3)}) = d(S^3_{-n}(K), \t_i) - d(S^3) = d(S^3_{-n}(K), \t_i).
\]
Finally, using that $d(\s)=0$, we compute 
\begin{align*}
d(\hat \s) &= \dsh(\hat X_{-n} \setminus \intB, \hat \s) - 1\\
&= (\dsh(X\setminus \intB, \s) -1) + \dsh(X_{-n}(K)^\circ, \s_i)\\
&= d(\s) + \gr(F^+)\\
&= d(S^3_{-n}(K), \t_i)
\end{align*}
proving that $\hat X_{-n}$ is of relative Ozsv\'ath-Szab\'o simple type.

Now, to check when condition \eqref{eq:HFp} is satisfied, we consider the map induced by $X_{-n}(K)^\circ$ on $\HFi$, which is nontrivial for $n > 0$ and for all \spinc structures~\cite[Proposition 9.4]{OSzAbsolutely}.

Under our assumption that $n>0$ there is a single tower in both the source and the target of $F^+$. 
Thus, $F^+$ is nonzero on the bottom element if and only if its grading shift $\gr(F^+)$ coincides with $d(S^3_{-n}(K), \hat\s|_{S^3_{-n}(K)})$.

If we choose the \spinc structure $\s_{-i}$ on $X_{-n}(K)^\circ$, we compute
\[
\gr(F^+) = \frac{c_1(\s_{-i})^2 - 2\chi - 3\sigma}{4} = \frac{-\frac{(-n+2i)^2}n - 2\cdot1 - 3\cdot(-1)}4 = \frac{n-(n-2i)^2}{4n}.
\]

On the other hand, using $S^3_{-n}(K) = -S^3_{n}(\mK)$ and $d(-Y,\t) = -d(Y, \t)$, we get $d(S^3_{-n}(K), \t_{-i}) = -d(S^3_{n}(\mK), \t_{-i})$, and by Ni-Wu's formula~\cite[Proposition 1.6 and Remark 2.10]{NiWu}, which holds for $n>0$, we get
\[
d(S^3_{-n}(K), \t_{-i}) = \frac{n-(n-2[-i])^2}{4n} + \max\set{V_{[-i]}(\mK), V_{n-[-i]}(\mK)},
\]
where $[-i]\in \{0, 1, \dots, n-1\}$ denotes the reduction of $-i \pmod n$.

By imposing $\gr(F^+) = d(S^3_{-n}(K), \t_{-i})$, we get
\begin{equation}
\label{eq:F+gr=d}
\frac{(n-2[-i])^2}{4n} = \frac{(n-2i)^2}{4n} + \max\set{V_{[-i]}(\mK), V_{n-[-i]}(\mK)}.
\end{equation}

Since the $V_k$'s are non-negative, for $i<0$ or $i>n$ the right hand side of Equation \eqref{eq:F+gr=d} is stricly bigger the the left hand side, so the equality cannot hold.
For $0 \leq i \leq n$, the terms $\frac{(n-2[-i])^2}{4n}$ and $\frac{(n-2i)^2}{4n}$ coincide, so Equation \eqref{eq:F+gr=d} reduces to
\[
\max\{V_{n-i}(\mK), V_{n-(n-i)}(\mK)\} = 0.
\]
Since the $V_k$'s are non-increasing, we can rewrite it as
\[
V_{\min\{i, n-i\}}(\mK)=0.
\]
This is possible if and only if $\nu^+(\mK) \leq n/2$. So we can satisfy the third bullet point by choosing $n\ge 2\nu^+(\mK)$. 
When $X$ is of Ozsv\'ath-Szab\'o simple type, we can choose any $n \geq \max\set{1, 2\nu^+(\mK)}$.
The other case was $[\Sigma]^2>0$: recall that in such a case we previously chose $0<n\le[\Sigma]^2$, so we understand the necessity of our hypothesis $2\nu^+(\mK)\le [\Sigma]^2$.

For any such an $n$, we will try to maximize the left hand side of Equation \eqref{eq:adj-nu+-pf-1}. Towards that aim, choose $\hat\s = \s \# \s_{-i}$ on $\hat X_{-n}$, then
\[
\scal{c_1(\hat\s)}{[\hat\Sigma_{-n}]} = \scal{c_1(\s)}{[\Sigma]} + \scal{c_1(\s_{-i})}{H} = \scal{c_1(\s)}{[\Sigma]} + (-n+2i).
\]
Therefore, Equation \eqref{eq:adj-nu+-pf-1} becomes
\begin{equation}
\label{eq:adj-nu+-pf-2}
\scal{c_1(\s)}{[\Sigma]} + [\Sigma]^2 - 2n +2i \leq 2g(\Sigma)-2.
\end{equation}

To maximize the left hand side of Equation \ref{eq:adj-nu+-pf-2} we should choose $i = n-\nu^+(\mK)$, which yields precisely the inequality \eqref{eq:adjunction-nu+}.

We now consider the case when $[\Sigma]^2=0$ (and when $X$ is not of Ozsv\'ath-Szab\'o simple type). In such a case, our hypothesis $[\Sigma]^2 \geq 2\nu^+(\mK)$ forces $\nu^+(\mK) = 0$, and therefore $V_0(\mK) = 0$ too.
Let $X_0(K)$ be the trace of the $0$-surgery, endowed with the spin structure $\s_0$, characterised by $c_1(\s_0)^2 = 0$, and let $\t_0$ denote the restriction of $\s_0$ to $S^3_0(K)$. By \cite[Proposition 22]{Gainullin},
\[
\HFp(S^3_0(K), \t_0) = \Tower^+_{-\frac12+2V_0(\mK)} \oplus \Tower^+_{\frac12-2V_0(K)} \oplus \HFred(S^3_0(K), \t_0),
\]
where $\Tower_d$ denotes a copy of the tower $\F[U, U^{-1}]/(U\F[U])$ starting in minimal degree $d$, and $\HFred$ is the reduced Heegaard Floer group from \cite[p.1068]{HolDisk}.

Since $U$ decreases the Maslov grading by $2$, the homogeneous elements of $\Tower^+_{-\frac12+2V_0(\mK)}$ have Maslov grading in $2\Z-\frac12$, while those of $\Tower^+_{\frac12-2V_0(K)}$ have Maslov grading in $2\Z+\frac12$.

By \cite[Proposition 9.3]{OSzAbsolutely}, the map
\[
F^+_{\s_0} \colon \HFp(S^3) \to \HFp(S^3_0(K), \t_0)
\]
is non-trivial, so its image must be one of the two towers. Since the grading shift is
\[
\gr(F^+_{\s_0}) = \frac{c_1^2(\s_0) - 2\chi(X_0(K)) - 3\sigma(X_0(K))}{4} = -\frac12 \in 2\Z-\frac12,
\]
we deduce that $\im (F^+_{\s_0}) = \Tower^+_{-\frac12+2V_0(\mK)}$.
Then, the condition $V_0(\mK) = 0$ and the fact that $\gr(F^+_{\s_0}) = -\frac12$ guarantee that the bottom element of $\HF^+(S^3)$ is sent to the bottom element of $\Tower^+_{-\frac12+2V_0(\mK)} = \Tower^+_{-\frac12}$. Thus, condition \eqref{eq:HFp} is satisfied. %
\end{proof}

Hom and Wu proved in  \cite{HW} that
\[
0 \leq \nu^+(\mK) \leq g_4(K),
\]
so the inequality \eqref{eq:adjunction-nu+} is stronger than \eqref{eq:adj-g4}. On the other hand,  the inequality \eqref{eq:adjunction-nu+} says nothing about null-homologous surfaces, because $\nu^+(\mK) \geq 0$ is already automatic.

\begin{remark}
Theorem~\ref{thm:adjunction-nu+} should be compared to another relative adjunction inequality, due to Hedden and Raoux \cite{HR}. They proved that, if $X$ is a smooth, oriented four-manifold with boundary $Y$ and $\Sigma \subset X$ a properly smoothly embedded surface such that the relative element $F_{X, \s} \in \widehat{HF}(Y)$ is nontrivial, then
$$2 \tau(K) + \langle c_1(\s), [\Sigma]\rangle + [\Sigma]^2\leq 2g(\Sigma).$$
When $Y=S^3$, the hypotheses are only satisfied for negative definite $4$-manifolds, and we get the inequality 
$$ 2 \tau(K) + \| [\Sigma]\|_{L^1} + [\Sigma]^2\leq 2g(\Sigma),$$
which was previously proved by Ozsv\'ath and Szab\'o \cite{os-tau}. 
\end{remark}

\begin{remark}
\label{rem:mr}
In \cite{MR}, Mrowka and Rollin proved a genus bound for surfaces in four-manifolds with contact boundary. Specializing to the case when the boundary is $S^3$ (with its standard tight contact structure $\xi$), their result says that if $\Xo$ has non-trivial Seiberg-Witten invariant in the \spinc structure $\s$, relative to $\xi$, and $\Sigma \subset \Xo$ is a smoothly, properly embedded surface, then
\begin{equation}
\label{eq:MR}
 \langle c_1(\s), [\Sigma]\rangle + [\Sigma]^2\leq 2g(\Sigma) - 1 - \slbar(K),
 \end{equation}
where $\slbar(K)$ is the maximal self-linking number of transverse knots in the isotopy class of $K$.
However, the Seiberg-Witten invariant of $\Xo$ is different from that of $X$. The Seiberg-Witten invariant of $X$ is conjectured to be the same as the Ozsv\'ath-Szab\'o mixed invariant. The Seiberg-Witten invariant of $\Xo$  is expected to correspond to the image of the cobordism map on $\HFp$ rather than the mixed map, and therefore should be non-trivial only when $X$ is negative definite. In fact, a version of \eqref{eq:MR} using the map on $\HFp$ was proved by Wu \cite{Wu}.

We note that the work in \cite{MR} was preceded by genus bounds in Stein manifolds \cite{AM, LM}. Another related result is a version of the symplectic Thom conjecture for manifolds with boundary, which was proved in \cite[Theorem 7.2.3]{Bowden} and \cite[Theorem 1.2]{GK}: If $X$ is a symplectic four-manifold with convex boundary, and $\Sigma \subset X$ is a symplectic surface with boundary a transverse knot $K \subset \del X$, then $\Sigma$ is genus minimizing in its relative homology class.\end{remark}

\subsection{Applications}

\begin{corollary}
\label{cor:spinsymp}
Let $X$ be the $\KT$ surface. If $\nu^+(\mK) = 0$, then the knot $K$ does not bound a positive self-intersection surface with genus $g(\Sigma)\leq1$ in $\Xo$.
\end{corollary}

\begin{proof}
The $\KT$ surface has one basic class, namely  the \spinc structure satisfying $c_1(\s)=0$. If $g(\Sigma) = 1$, we can apply \eqref{eq:adjunction-nu+} to this case. If $g(\Sigma) = 0$, we stabilize $\Sigma$ once before applying \eqref{eq:adjunction-nu+}.
\end{proof} 

\begin{example}
By Theorem \ref{thm:adjunction-g4}, if $K$ is one of the trefoils or the figure-eight knot, then $K$ does not bound a disk $\Delta$ with $[\Delta]^2 > 0$ in the $\KT$ surface. 

Since the left-handed trefoil and the figure-eight knots have $\nu^+=0$, we get better bounds from Theorem \ref{thm:adjunction-nu+}. Indeed, by Corollary \ref{cor:spinsymp}, the right-handed trefoil or the figure-eight knot do not even bound a genus-1 surface $\Sigma$ with $[\Sigma]^2 > 0$ in $\KT$.

More generally, by Theorem \ref{thm:adjunction-g4} we see that for $p,q>0$ and coprime, if the torus knot $T_{p,q}$ bounds a surface $\Sigma$ of genus $g$ inside $\KT$, then $[\Sigma]^2 \leq 2g-2+(p-1)(q-1)$. However, by Theorem \ref{thm:adjunction-nu+} we get a better bound, namely $[\Sigma]^2 \leq 2g-2$ (provided $g > 0$).
\end{example}

\begin{remark}
The applications in this section were stated for the $\KT$ surface, but they apply just as well to any $4$-manifold with $b_2^+ > 0$ that admits a basic class $\mathfrak{s}$ with $c_1(\mathfrak{s}) = 0$; e.g. to any homotopy $\KT$ surface (see \cite{MorganSzabo}). 
\end{remark}

\subsection{An application using the Bauer-Furuta invariants}

\begin{proposition}
\label{prop:bf}
Let $X$ be a smooth closed $4$-manifold with $b_2^+(X) \equiv 3 \!\! \pmod{4}$, admitting a \spinc structure $\s$ such that $d(\s)=0$ and $SW_{X,\s}$ odd. Let $X'$ be another $4$-manifold with the same properties. Suppose that a knot $K \subset S^3$ is such that the mirror $\mK$ bounds a smooth, properly embedded disk $\Delta \subset \Xo$ with $[\Delta]^2 \geq 0$ and $[\Delta] \neq 0$. Then $K$ is not H-slice in $X'$.
\end{proposition}

\begin{proof}
Since $d(\s)=0$ and $b_2^+(X) \equiv 3 \!\! \pmod{4}$, the Bauer-Furuta invariant 
$$ \BF_{X,\s} \in \pi_1^{\st}(S^0) \cong \Z/2$$
 is the mod $2$ reduction of the Seiberg-Witten invariant; see \cite[proof of Proposition 4.4]{Bauer}. By our assumption $SW_{X,\s}$ is odd, so $\BF_{X,\s}$ is the nontrivial element $\eta$ (the Hopf map) in $\pi_1^{\st}(S^0)$.  The same is true for $\BF_{X', \s'}$. Applying the connected sum formula (\cite[Theorem 1.1]{Bauer}), we obtain
\begin{equation}
\label{eq:BFXX}
 \BF_{X \# X', \s \# \s'} = \eta^2 \neq 0 \in  \pi_2^{\st}(S^0) \cong \Z/2.
 \end{equation}
 
Suppose that $K$ bounds an H-slice disk $\Delta'$ in $(X')^\circ$. By gluing $\Delta$ to $\Delta'$ we obtain a smoothly embedded sphere $S \subset X \# X'$ with $[S]^2 \geq 0$ and $[S]\neq 0$. From \eqref{eq:BFXX} and Theorem~\ref{thm:BFadj}, we derive a contradiction.
\end{proof}

{
\renewcommand{\thethm}{\ref{thm:symp}}
\begin{theorem}
Let $X$ and $X'$ be closed symplectic $4$-manifolds satisfying $b_2^+(X) \equiv b_2^+(X') \equiv 3 \!\! \pmod{4}$. Suppose that a knot $K \subset S^3$ is such that the mirror $\mK$ bounds a smooth, properly embedded disk $\Delta \subset \Xo$ with $[\Delta]^2 \geq 0$ and $[\Delta] \neq 0$. Then $K$ is not H-slice in $X'$.
\end{theorem}
\addtocounter{thm}{-1}
}

\begin{proof}
This follows from Proposition~\ref{prop:bf}, using the facts that the canonical class $\mathfrak{k}$ of a symplectic manifold satisfies $d(\mathfrak{k}) = 0$ and $SW_{X, \mathfrak{k}} =\pm 1$; cf.~\cite{Taubes}. 
\end{proof}

{
\renewcommand{\thethm}{\ref{cor:exotica}}
\begin{corollary}
There exist smooth, homeomorphic four-manifolds $X$ and $X'$ and a knot $K \subset S^3$ that is H-slice in $X$ but not in $X'$. For example, one can take $$X=\#3\CP \# 20 \bCP, \ \ X' = K3 \# \bCP,$$ and $K$ to be the right-handed trefoil. 
\end{corollary}
\addtocounter{thm}{-1}
}

\begin{proof}
Note that $X =\#3\CP \# 20 \bCP$ and $X' = K3 \# \bCP$ are simply connected, smooth, closed four-manifolds with the same intersection form, so they are homeomorphic by Freedman's theorem \cite{Freedman}.

Let $K$ be the right-handed trefoil. Then $K$ is H-slice in $X$ because it is already H-slice in $\CP$; see Example~\ref{ex:RHT}.

Recall from Example~\ref{ex:LHTinK3} that the left-handed trefoil $\mK$ bounds a slice disk $\Delta$ in K3 with $[\Delta]^2 =0$ and $[\Delta] \neq 0$. Since both K3 and its blow-up $X'=K3 \# \bCP$ are complex projective surfaces (hence symplectic), we can apply Theorem \ref{thm:symp} to deduce that $K$ is not H-slice in $X'$.
\end{proof}

\section{Relative Donald-Vafaee obstructions}
\label{sec:DVT}
The following theorem is a generalization of the results of \cite{DonaldVafaee} and \cite{Truong} (who did the case $X=S^3$) to spin 4-manifolds.

{
\renewcommand{\thethm}{\ref{thm:DV}}
\begin{theorem}
Let $K \subset S^3$ be an $H$-slice knot in a closed spin 4-manifold $X$, and let $W$ be a spin 2-handlebody with $\de W = S^3_0(K)$. If $b_2(X) + b_2(W) \neq 1, 3, 23$, then
\[
b_2(X) + b_2(W) \geq \frac{10}{8} \cdot |\sigma(X) - \sigma(W)| + 5.
\]
\end{theorem}
\addtocounter{thm}{-1}
}

\begin{proof}
The proof is analogous to \cite{DonaldVafaee}.
A neighborhood of the slice disk $\Delta$ in $\Xo$ together with the removed $\intB$ gives an embedding of $X_0(m(K))$, the trace of the $0$-surgery on the mirror $m(K)$, inside $X$. Note that $\del X_0(m(K)) = -S^3_0(K)$ and we have a splitting $X = X_0(m(K)) \cup V$.
The manifold $V$ is spin (by restricting the spin structure on $X$). We will show in the next paragraph that the map
\[
H^1(V;\F_2) \to H^1(\de V; \F_2)
\]
is surjective, which implies that both spin structures on $S^3_0(K)$ extend to $V$.

Consider the following portion of the Mayer-Vietoris long exact sequence in homology associated to $X = X_0(m(K)) \cup V$:
\[
H_2(X; \F_2) \xrightarrow{f} H_1(S^3_0(K); \F_2) \xrightarrow{g} H_1(X_0(m(K)); \F_2) \oplus H_1(V; \F_2).
\]
Since $H_1(S^3_0(K); \F_2)$ is 1-dimensional, generated by the meridian $\mu$, and $H_1(X_0(m(K)); \F_2) = 0$, we get
\[
H_2(X; \F_2) \xrightarrow{f} \F_2\langle\mu\rangle \xrightarrow{g} H_1(V; \F_2).
\]

If $F$ is a closed surface in $X^\circ \subset X$ transverse to $\Delta$, then $F \cap S^3_0(K)$ consists of copies of the meridian $\mu$, and each copy of $\mu$ corresponds to an intersection point $F \cap \Delta$. Thus, $f([F]) = ([F]\cdot[\Delta])\mu \pmod2$. Since $[\Delta] = 0$ (with $\Z$ coefficients, and hence with $\F_2$ coefficients), the map $f$ is vanishing, and by exactness
\[
g \colon H_1(S^3_0(K);\F_2) \to H_1(V;\F_2)
\]
is injective.
By taking duals, the restriction map
\[
H^1(V;\F_2) \to H^1(\de V; \F_2)
\]
is surjective. Thus, both spin structures on $S^3_0(K)$ extend to $V$.

We now consider $Z = (-V) \cup W$. The spin structure on $W$ extends to a spin structure on $Z$. Moreover, Mayer-Vietoris shows that $b_2(Z) = b_2(V) + b_2(W) - 1 = b_2(X) +  b_2(W) -1$. By Novikov's additivity theorem $\sigma(Z) = -\sigma(V) + \sigma(W) = -\sigma (X) + \sigma(W)$.
Thus, by \cite{HLSX}, we get
\[
b_2(X) + b_2(W) \geq \frac{10}{8} \cdot |\sigma(X) - \sigma(W)| + 5. \qedhere
\]
\end{proof}

\begin{figure}
\begingroup%
  \makeatletter%
  \providecommand\color[2][]{%
    \errmessage{(Inkscape) Color is used for the text in Inkscape, but the package 'color.sty' is not loaded}%
    \renewcommand\color[2][]{}%
  }%
  \providecommand\transparent[1]{%
    \errmessage{(Inkscape) Transparency is used (non-zero) for the text in Inkscape, but the package 'transparent.sty' is not loaded}%
    \renewcommand\transparent[1]{}%
  }%
  \providecommand\rotatebox[2]{#2}%
  \newcommand*\fsize{\dimexpr\f@size pt\relax}%
  \newcommand*\lineheight[1]{\fontsize{\fsize}{#1\fsize}\selectfont}%
  \ifx\svgwidth\undefined%
    \setlength{\unitlength}{206.58172367bp}%
    \ifx\svgscale\undefined%
      \relax%
    \else%
      \setlength{\unitlength}{\unitlength * \real{\svgscale}}%
    \fi%
  \else%
    \setlength{\unitlength}{\svgwidth}%
  \fi%
  \global\let\svgwidth\undefined%
  \global\let\svgscale\undefined%
  \makeatother%
  \begin{picture}(1,0.34853191)%
    \lineheight{1}%
    \setlength\tabcolsep{0pt}%
    \put(0,0){\includegraphics[width=\unitlength,page=1]{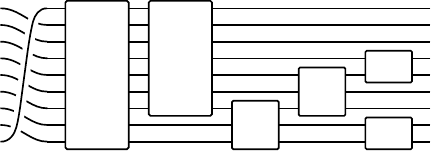}}%
    \put(0.18045842,0.15977397){\color[rgb]{0,0,0}\makebox(0,0)[lt]{\lineheight{0}\smash{\begin{tabular}[t]{l}$+1$\end{tabular}}}}%
    \put(0.38583184,0.19987073){\color[rgb]{0,0,0}\makebox(0,0)[lt]{\lineheight{0}\smash{\begin{tabular}[t]{l}$-1$\end{tabular}}}}%
    \put(0.55795432,0.04437371){\color[rgb]{0,0,0}\makebox(0,0)[lt]{\lineheight{0}\smash{\begin{tabular}[t]{l}$-1$\end{tabular}}}}%
    \put(0.71239478,0.12261132){\color[rgb]{0,0,0}\makebox(0,0)[lt]{\lineheight{0}\smash{\begin{tabular}[t]{l}$-1$\end{tabular}}}}%
    \put(0.86886967,0.17835555){\color[rgb]{0,0,0}\makebox(0,0)[lt]{\lineheight{0}\smash{\begin{tabular}[t]{l}$-1$\end{tabular}}}}%
    \put(0.86992606,0.02383649){\color[rgb]{0,0,0}\makebox(0,0)[lt]{\lineheight{0}\smash{\begin{tabular}[t]{l}$-1$\end{tabular}}}}%
  \end{picture}%
\endgroup%

\caption{The closure of the braid above is the topologically slice knot from \cite[Figure 3]{DonaldVafaee}.
The number within each box indicates the number of full twists performed there.
}
\label{fig:DV}
\end{figure}

\begin{remark}
A version of Theorem~\ref{thm:DV} still holds if we allow $W$ to be any spin $4$-manifold with $\del W = S^3_0(K)$ instead of a $2$-handlebody. In that case, we obtain the weaker inequality
\[
b_2(X) + b_2(W) \geq \frac{10}{8} \cdot |\sigma(X) - \sigma(W)| + 4.
\]
Indeed, the same proof applies, except that $b_2^+(Z)$ can now be either $b_2^+(V) + b_2^+(W)$ or $b_2^+(V) + b_2^+(W) - 1$.
\end{remark}

{
\renewcommand{\thethm}{\ref{cor:DV}}
\begin{corollary}
Let $\KDV$ be the topologically slice knot in \cite[Figure 3]{DonaldVafaee}, which is the closure of the braid in Figure \ref{fig:DV}. Then $\KDV$ is not $H$-slice in the $\KT$ surface.
\end{corollary}
\addtocounter{thm}{-1}
}

\begin{proof}
It is shown in \cite[Example 3.4]{DonaldVafaee} that $S^3_0(\KDV)$ bounds a 2-handlebody $W$ with $b_2(W) = 21$ and $\sigma(W) = 16$.
Since $b_2(\KT) = 22$ and $\sigma(\KT) = -16$, if the knot were $H$-slice in $\KT$ we would get
\[
22+21 \geq \frac{10}8 \cdot 32 + 5,
\]
which is false.
\end{proof}

\begin{remark}
What is really needed for Corollary~\ref{cor:DV} is the non-existence of a smooth, closed, simply connected $4$-manifold with intersection form $4(-E_8) \oplus 5 \left(\begin{smallmatrix}
 0 & 1 \\ 1 & 0  \end{smallmatrix}\right)$. This does not quite follow from Furuta's $10/8 + 2$ theorem \cite{Furuta}, but it was known before the general $10/8 + 4$ theorem from \cite{HLSX}. Indeed, it is due to Furuta, Kametani and Matsue \cite{FKM2}.
\end{remark}

\section{Open problems}
\label{sec:open}

Theorems~\ref{thm:Arf}, \ref{thm:CN}, \ref{thm:symp} and \ref{thm:DV} provide obstructions for a knot to be H-slice in an indefinite $4$-manifold. For example, H-slice knots $K$ in the K3 surface must satisfy $\Arf(K)=0$ and $\sigma_K(\omega) \in [-6, 38]$ for all $\omega \in \Se$. Theorems~\ref{thm:FK}, \ref{thm:rokhlin} and \ref{thm:adjunction-nu+} give additional constraints on the homology class of a slice disk in such a four-manifold. One can ask whether, by these and other methods, one can rule out all homology classes for a slice disk for a knot $K$ in an indefinite $4$-manifold, in the spirit of Yasuhara's proof that $T_{2,-15}$ is not slice in $\CP$ \cite{YasuharaExample}. In particular, we raise the following:  

\begin{question}
Is there a knot that is not slice in the K3 surface?
\end{question}

Also, we can ask whether the analogue of Corollary~\ref{cor:exotica} holds for slice knots (instead of H-slice knots).

\begin{question}
Can the set of slice knots detect exotic smooth structures? In other words, do there exist smooth, homeomorphic four-manifolds $X$ and $X'$ and a knot $K \subset S^3$ that is slice in $X$ but not in $X'$?
\end{question}

Finally, in view of Corollaries~\ref{cor:exotica} and \ref{cor:DV}, we propose the following problem.

\begin{question}
Is it true that for every closed $4$-manifold $X$, there is a knot $K$ that is topologically but not smoothly H-slice in $X$?
\end{question}

\bibliographystyle{custom}
\bibliography{topology}
\end{document}